\documentclass[12pt,reqno]{amsart}

\usepackage{amsfonts,appendix,color, latexsym, amssymb, amsmath, fullpage, bbm}

\title[Truncations of Kronecker Products]{Local Spectrum of Truncations of Kronecker Products of Haar Distributed Unitary Matrices}
\author{Brendan Farrell}
\address{Computing and Mathematical Sciences, California Institute of Technology, Pasadena, CA 91125, U.S.A. }
\email{farrell@cms.caltech.edu}
\author{Raj Rao Nadakuditi}
\address{Department of Electrical Engineering and Computer Science,
University of Michigan, Ann Arbor, MI 48109, USA}
\email{rajnrao@eecs.umich.edu}
\date{\today}
\newcommand{\C}{\mathbb C}
\newcommand{\E}{\mathbb E}
\def \P {\mathbb{P}}

\newtheorem{theorem}{Theorem}[section]

\newtheorem{lemma}{Lemma}[section]

\def\C {\mathbb{C}}
\def\P {\mathbb{P}}
\def\NN {\mathbb{N}}

\def\E {\mathbb{E}}
\def\U {\mathcal{U}}
\def\W {\mathcal{W}}

\def\N {\mathcal{N}}
\def\tr {\textnormal{tr}}

\def\Wstar {\mathcal{W}^*}

\def\Rj {R_j}

\def\Qj {Q_j}
\def\uj {u_j}
\def\ujs {u_j^*}

\def\onk {\frac{1}{n^k}}
\def\O {\mathcal{O}}
\def\o {\mathit{o}}
\def\moz {\frac{-1}{z}}

\def\Uma {U-\gamma Uvv^*}

\def\ok {\otimes k}
\def\vvs {vv^*}
\def\oh {\frac{1}{2}}
\def\oz {\frac{1}{z}}
\def\cz {c_0}
\def\tr {\textnormal{tr}}
\def\Qjj {Q_{j,j}}
\def\Pjj {P_{j,j}}
\def\region {z\in \C:\Re z\in [\lambda_-+\kappa,\lambda_+-\kappa],\;\Im z\in (\eta,\cb]}
\def\interval{[\lambda_-+\kappa, \lambda_+-\kappa]}
\def\lm {\lambda_-}
\def\lp {\lambda_+}
\def\ez {\eta_0}
\def\step {E+i(\eta_0-n^{-2})}
\def\overm {\frac{1-p}{zm(z)}}
\def\overmM {\frac{1-p}{zm_M(z)}}

\def\ca {c_a}
\def\cb {c_b}
\def\cc {c_c}

\def\An {A_n}
\def\Bn {B_n}
\def\Un {U_n}
\def\Ub {\mathbb{U}}
\def\nn {\langle n\rangle}

\def\nnku {\nn^k_u}
\def\Wo {W_1}
\def\Wos {W_1^*}

\def\wj {w_j}
\def\wjs {w_j^*}

\def\sumu {\sum_{j\in\nnku}}
\def\sumnn {\sum_{j\in\nn^k}}
\def\sumc {\sum_{j\in\nn\backslash\nnku}}

\begin{document}
\begin{abstract}
We address the local spectral behavior of the random matrix 
\begin{equation*}
\Pi_1 U^{\ok} \Pi_2 U^{\ok *} \Pi_1,
\end{equation*}
where $U$ is a Haar distributed unitary matrix of size $n\times n$, the factor $k$ is at most $c_0\log n$ for a small constant $c_0>0$, 
and $\Pi_1,\Pi_2$ are arbitrary projections on $\ell_2^{n^k}$ of ranks proportional to $n^k$. 
We prove that in this setting the $k$-fold Kronecker product behaves similarly to the well-studied case when $k=1$. 
\end{abstract}
\maketitle


\vspace{.2in}


\section{Introduction}
A fundamental question in matrix analysis is: How are the eigenvalues of the sum or product of two matrices related to the eigenvalues of the individual matrices? This simple question has a complicated solution because the answer depends not just on the eigenvalues but on the relationship between the eigenspaces of the individual matrices. 

However, if one lets the eigenspaces of one of the matrices be isotropically random relative to the other, then in the limit of large matrices, we can make analytical progress. Specifically, if $\Un$ is a Haar distributed unitary matrix and if $\{\An\}_{n\in\NN}$ and $\{\Bn \}_{n\in\NN}$ are two sequences of bounded self-adjoint
matrices ($\An,\Bn\in\C^{n\times n}$) then the spectral distribution of
\begin{equation*}
\An\Un\Bn\Un^*\An^*\;\;\;\textnormal{and}\;\;\;\An +\Un\Bn\Un^*
\end{equation*}
in the limit of large matrices is completely characterized by an additive (or multiplicative, respectively) `free convolution' \cite{VDN92} operation involving only the individual limiting spectral distributions of $\{\An\}$ and $\{\Bn\}$~\cite{Voi91}. 

Let $\Ub(n)$ denote the group of unitary matrices of size $n\times n$, and consider the matrix
\begin{equation}
C:= \Pi_1 U \Pi_2 U^* \Pi_1\label{eq:nokronecker},
\end{equation}
where $\Pi_{1}$ and $\Pi_{2}$ are two arbitrary orthogonal projections on $\ell_2^{n}$ of ranks, say,
$pn$ and $qn$ respectively and  $U$ has Haar distribution (or uniform distribution) on $\Ub(n)$. 
Then, a consequence of Voiculescu's theorem is that the limiting spectral measure of $C$ is given by the free 
multiplicative convolution of the limiting spectral measures of the individual projection matrices.  
The spectral measures of projection matrices are Bernoulli distributions, and their free multiplicative convolution  is $f_M$, which we define shortly. 
In particular, we first define the empirical distribution function $F$ by
\begin{equation*}
F(x)=\frac{1}{n}\sharp\{\lambda_i(\Pi_1 U \Pi_2 U^* \Pi_1)\leq x\}.
\end{equation*}
Then as $n$ tends to infinity, $F$ converges almost surely to the distribution
\begin{equation*}
f_M(x)dx:=(1-\min(p,q))\delta_0(x)+(\max(p+q-1,0))\delta_1(x)+\frac{\sqrt{(\lp-x)(x-\lm)}}{2\pi x(1-x)} I_{[\lm,\lp]}(x)dx,
\end{equation*}
where
\begin{equation*}
\lambda_{\pm}:=p+q-2pq\pm\sqrt{4pq(1-p)(1-q)}.
\end{equation*}

We use $f_M$ to denote that this density is the limiting density for matrices used for multivariate analysis of
variance in statistics (MANOVA). The density $f_M$ was first determined in this setting by Wachter~\cite{Wac80}, though it appears
earlier in work by Kesten on regular graphs~\cite{Kes59}. 
The matrix~\eqref{eq:nokronecker} has been studied by a number of authors. 
In the free probability community \cite{CC04,Col05} provide extensive results and treat this matrix in the context of the Jacobi ensemble and 
classical random matrix theory.  
The paper~\cite{ZS00} studies the absolute values of the eigenvalues of $\Pi_1 U \Pi_2$ and has led to stronger results on the 
eigenvalue distribution of such matrices, see~\cite{DJL12}. 
In the statistics community, Tracy-Widom behavior of the largest eigenvalue of the matrix~\eqref{eq:nokronecker} was established in~\cite{Joh08}. 

The present  work is motivated by the question of whether the same limit distribution arises  when $U$ is not uniformly distributed on $\Ub(n)$. 
We heuristically conjecture that the spectral distributions will be close when $\Un$ is distributed such that a `typical' realization is `close' to a typical Haar distributed unitary matrix. Since a typical Haar distributed  matrix in $\mathbb{U}(n^k)$ 
has entries with magnitude $O(1/\sqrt{n^k})$, 
a `sufficiently random' $U$ with entries  having magnitude $O(1/\sqrt{n^k})$ might  
exhibit the same limiting distribution. This paper is a first step in the general program of trying to quantify these notions of closeness.

To that end, we consider unitary matrices that are formed from the Kronecker (or tensor) product of uniformly distributed random unitary matrices. In particular we consider unitary matrices that are constructed as follows.
Let $n,k$ be integers  and  $\Pi_{1}$ and $\Pi_{2}$  arbitrary orthogonal projections on $\ell_2^{n^k}$ of ranks
$pn^k$ and $qn^k$ respectively.
Let $U$ have Haar distribution on $\Ub(n)$, and consider the matrix
\begin{equation}
\Pi_1 U^{\ok} \Pi_2 U^{\ok *} \Pi_1. \label{equation:matrix}
\end{equation}
We will show that, for large $n$ and appropriate $k$, 
the eigenvalues of~\eqref{equation:matrix} are distributed similarly to those of~\eqref{eq:nokronecker}. 
Note that 
Tensor products of random unitary matrices have been recently studied by several authors 
\cite{tkocz2011tensor,belinschi2012eigenvectors,collins2012towards}. 
In particular, in~\cite{tkocz2011tensor} it is shown that when $k=2$ and $n$ tends to infinity the spacing between eigenvalues of the 
models~\eqref{eq:nokronecker} and~\eqref{equation:matrix} differ qualitatively. 
Thus, while the eigenvalue distributions of the matrices considered here are preserved by taking a tensor product, 
the spacings between the eigenvalues of $U^{\otimes k}$ are not for $k=2$. 
The same is presumably also true for larger powers of $k$. 

One motivation for studying $U^{\otimes k}$ is that it has more structure and less randomness than a Haar distributed unitary matrix of the same dimensions, 
yet behaves similarly. 
One area where such randomness reduction is of interest is quantum information theory \cite{musz2013unitary}. 
For example,~\cite{GAE07} addresses subsets of $\mathbb{U}(n)$ that can be used to approximate the expectation of a function of $U$. 
The present work shows that in the projection setting, Haar distribution for $U$ in high dimensions is close to the distribution generated using 
much less randomness and requiring less computational complexity. 
Thus, an expectation of $U^{\otimes k}$, where $U$ has Haar distribution on $\mathbb{U}(n)$ could be used to approximate an expectation of 
$U'$, where $U'$ has Haar distribution on $\mathbb{U}(n^k)$. 

We state our precise result next and then provide a discussion of further topics and an outline of our approach. 

\section{Main result}
First, let us introduce the Stieltjes transform of a distribution $F$:
\begin{equation*}
m_F(z):=\int \frac{1}{x-z}dF(x)
\end{equation*}
for $z\in\C^+:=\{z\in\C:\;\Im z>0\}$. Note the Stieltjes transform of a distribution is an analytic map
from $\C^+$ to $\C^+$.
The Stieltjes transform of $f_M$ is available~\cite{CC04}; it is
\begin{equation*}
m_M(z):=\frac{z+(p+q-2)+\sqrt{z^2-2(p+q-2pq)z+(p-q)^2}}{2z(1-z)}.
\end{equation*}

We are now ready to state the main result.
\begin{theorem}\label{theorem:main}
For $E\in [\lm,\lp]$,  let $\N(E,\eta)$ denote the number of eigenvalues of~\eqref{equation:matrix} in $[E-\frac{\eta}{2},E+\frac{\eta}{2}]$.
Assume that $0<\cz<\frac{1}{2}$ and $k\leq \cz \log n$,
and set
\begin{equation*}
m(z):=\frac{1}{n^k}\tr \left(\Pi_1 U^{\ok} \Pi_2 U^{\ok *} \Pi_1-zI\right)^{-1}.
\end{equation*}
There exist absolute constants $C,\rho>0$ such that for all $s>0$ and $\alpha,\beta>0$ satisfying $\alpha+2\beta=\frac{1}{2}-\cz$, if
\begin{equation}
\eta:=\frac{\rho^{1/4}\log^{\frac{s}{2}+\frac{5}{2}} n}{n^\beta},\label{equation:defeta}
\end{equation}
then for all $\kappa>0$
\begin{equation}
\P\left(\sup_{E\in \interval}   |m(E+i\eta)-m_M(E+i\eta)|>\frac{C}{n^\alpha\kappa^{2}}   \right)\leq 2n^{k+2}e^{-\log^s n}\label{equation:maineqprop}
\end{equation}
and
\begin{equation}
\P\left(\sup_{E\in \interval}   \left|\frac{\N(E,\eta)}{\eta n^k}-f_M(E)\right|>\frac{C}{n^\alpha\kappa^{2}}   \right)\leq 2n^{k+2}e^{-\log^s n}.\label{equation:maineqthm}
\end{equation}
\end{theorem}

Our second theorem is a variation on the first.
We prove Theorem~\ref{theorem:main}; the proof of Theorem~\ref{theorem:second} follows by the same arguments and
is sketched at the end of the paper.

\begin{theorem}\label{theorem:second}
Let $U$ be uniformly distributed on $\Ub(n_1)$, let $V$ be an arbitrary element of $\Ub(n_2)$ and $\Pi_1,\Pi_2$
arbitrary orthogonal projections on $\ell^{n_1n_2}_2$ with respective ranks $pn_1n_2$ and $qn_1n_2$.
For $E\in [\lm,\lp]$,  let $\N(E,\eta)$ denote the number of eigenvalues of
\begin{equation}\label{eq:UV}
\Pi_1 (U\otimes V)\Pi_2 (U\otimes V)^* \Pi_1
\end{equation}
in $[E-\frac{\eta}{2},E+\frac{\eta}{2}]$, and set
\begin{equation*}
m(z):=\frac{1}{n_1n_2}\tr \big(\Pi_1 (U\otimes V)\Pi_2 (U\otimes V)^* \Pi_1   -zI\big)^{-1}.
\end{equation*}
There exist absolute constants $C,\rho>0$ such that for all $s>0$ and $\alpha,\beta>0$ satisfying $\alpha+2\beta=\frac{1}{2}$, if
\begin{equation*}
\eta:=\frac{\sqrt{\rho}\log^{\frac{s}{2}+4} n_1}{n_1^\beta},
\end{equation*}
then for all $\kappa>0$
\begin{equation*}
\P\left(\sup_{E\in \interval}   |m(E+i\eta)-m_M(E+i\eta)|>\frac{C}{n_1^\alpha\kappa^{2}}   \right)\leq 2n_1^{2}e^{-\log^s n_1}
\end{equation*}
and
\begin{equation*}
\P\left(\sup_{E\in \interval}   \left|\frac{\N(E,\eta)}{\eta n_1n_2}-f_M(E)\right|>\frac{C}{n_1^\alpha\kappa^{2}}   \right)\leq 2n_1^{2}e^{-\log^s n_1}.
\end{equation*}
\end{theorem}

\subsection{Remarks}
We point out several areas for further exploration. 
It is unclear what happens to the eigenvalue distribution of a matrix  of the form~\eqref{equation:matrix} when $n$ remains fixed but $k$ tends 
to infinity, or for a matrix of the form~\eqref{eq:UV} when the dimension of either the random or deterministic matrix is fixed and the other 
tends to infinity. 
Another variation on the work presented here would be to consider the Kronecker product of $k$ independent Haar distributed unitary matrices, 
possibly of varying dimensions, and determine their spectral behavior as a large unitary matrix as well as when truncated.

\subsection{Outline of the Approach}

We prove the first claim of each theorem, and the second then follows. 
Our approach is to show that $m(z)$ is the solution to a perturbed implicit equation;  
in particular, we use several resolvent identities to obtain the identity~\eqref{eq:implicitfinite} below.  
We then determine the expectation of the random term in this equation and show that it is also highly concentrated. 
The concentration result, Lemma~\ref{lemma:concentration}, is presented in Section~\ref{section:concentration} and is the most important part of the proof. 
Given this concentration we are able to isolate the perturbation and obtain the implicit equation~\eqref{eq:implicit}. 
We then show that this equation is stable, so that $m(z)$ is close to the solution to the unperturbed equation, which is $m_M(z)$.

\section{The Concentration Result}\label{section:concentration}
We begin with Theorem~\ref{theorem:main}. 
For two pairs of coordinate projections $P$ and $Q$ (also called ``diagonal projections'') 
and unitary matrices $W_1,W_2\in \Ub(n^k)$, we may write
$\Pi_1=W_1 PW_1^*$ and $\Pi_2=W_2QW_2^*$.
We then set 
\begin{equation*}
\U=U^{\ok}\;\;\text{and}\;\; \W= W_1^*U^{\ok}W_2^*
\end{equation*}
so that
\begin{equation*}
\Pi_1 U^{\ok} \Pi_2 U^{\ok *} \Pi_1= W_1P\W Q\Wstar P W_1^*,
\end{equation*}
which has the same eigenvalues as $P\W Q\Wstar P$.

We use $R(z)$ to denote the resolvent of our matrix of interest:
\begin{equation*}
R(z):=(P\W Q\Wstar P-zI)^{-1}.
\end{equation*}
For $n\in\mathbb{N}$ define
\begin{equation*}
\nn:=\{1,\ldots,n\}.  
\end{equation*}
The large matrix that we address  has dimensions $n^k\times n^k$, and we will use $\nn^k$ as our index set. 
The matrix $\U$ is indexed so that 
\begin{equation}\label{eq:index}
\U_{i,j}=U_{i_1,j_1}\cdots U_{i_k,j_k}.
\end{equation} 
We define 
\begin{equation*}\begin{array}{ll}
u_j&:=\;j^{th}\;\text{ column of }\U\\
\wj&:=\;j^{th}\;\text{ column of }W_1^*U^{\ok}W_2^*
\end{array}\end{equation*}
and 
\begin{equation*}
\Rj(z):=\left(  \sum_{k\neq j}Q_{k,k} P w_k w_k^* P-zI\right)^{-1}.
\end{equation*}

Using the definition of $\eta$ from~\eqref{equation:defeta} and fixed constants $\kappa,\cb>0$,
we define the region
\begin{equation}
\Omega:=\big\{\region\big\}.\label{equation:region}
\end{equation}
The constant $c_b$ will be chosen small enough to satisfy requirements for Lemma~\ref{lemma:startingpoint}

\begin{lemma}\label{lemma:concentration}
Assume all the hypotheses of Theorem~\ref{theorem:main}.
For any $s>0$, assume
\begin{equation*}
\eta\geq \frac{\sqrt{\rho}\log^{\frac{s}{2}+4}n}{n^\beta}.
\end{equation*}
Then for all $z\in\Omega$
\begin{equation}\label{equation:maxdelta}
\P\left( \max_{j\in\nn^k} |\wjs P\Rj(z) P \wj-\E \wjs P\Rj(z) P\wj  |>\frac{1}{n^{\alpha}}\right) \leq 2n^{k}e^{-\log^s n}.
\end{equation}
\end{lemma}

\begin{proof}
Since the columns of $U$ have the same distribution, $\ujs\Wo P\Rj(z) P\Wos\uj$ has the same distribution for all $j$.

For an arbitrary $j$ set
\begin{equation}
f(U):=\wjs P\Rj(z)P\wj- \E \wjs P\Rj(z) P\wj .
\end{equation}
We will use a concentration result on $\Ub(n)$ due to Chatterjee to show that $f(U)$ is concentrated around $0$.
Let $v$ be uniformly distributed on $S^{n-1}$, let $\phi$ be uniformly distributed on $[0,1]$, set $\gamma:=1-e^{2\pi i \phi}$ and 
set 
\begin{equation}
U':=U(I-\gamma \vvs)\;\;\;\text{  and  }\;\;\;\W':=W_1^*(U')^{\otimes k}W_2^*.\label{eq:Uprime}
\end{equation}
First we bound $(\E|f(U)-f(U')|^2)^\oh$ by showing that $|f(U)-f(U')|$ is small with high probability.
We set $T:=\W-\W'$  and define the resolvent $\Rj'(z)$ analogously to $\Rj(z)$.
The $j^{th}$ column of $\W'$ is denoted $w'_j$.  
Finally $\Qj$ denotes the matrix $Q$ with the entry $(j,j)$ set to $0$.
Thus
\begin{equation*}
\Rj(z)-\Rj'(z)=\Rj(z)[PT\Qj\Wstar P+P\W \Qj T^*P -PT\Qj T^*P ]\Rj'(z).
\end{equation*}
We will provide bounds for the following terms
\begin{small}\begin{eqnarray}
\lefteqn{|\wjs P\Rj(z)P\wj-(w'_j)^* P\Rj'(z)P \wj'|} \label{equation:termtobound}\\
&\leq& |\wjs P(\Rj(z)-\Rj'(z))P\wj|+2|(\wj-\wj')^* R(z)\wj|.\nonumber\\
&\leq&  2|\wjs P\Rj(z)PT\Qj\Wstar P\Rj'(z)P \wj|+|\wjs P\Rj(z)PT\Qj T^* P\Rj'(z)P \wj| \nonumber\\
&&+2|(\wj-\wj')^* R(z)\wj|.\label{equation:twotermsb}
\end{eqnarray}\end{small}
For $l=1,\ldots,k$, set $T_l:=W_1^*[U^{\otimes k-l}\otimes \gamma U\vvs\otimes (U(I-\gamma \vvs))^{\otimes l-1}]W_2,$ so that
\begin{eqnarray*}
T=W_1^*[U^{\otimes k}-(U(I-\gamma \vvs))^{\otimes k}]W_2=\sum_{l=1}^kT_l
\end{eqnarray*}
with the convention that $A^{\otimes 0}$ is the scalar $1$ for any matrix $A$.

For arbitrary $x,y$ and an arbitrary fixed $l$
we show that $|\langle x ,T_l y\rangle|$ is small with high probability.
We  set
\begin{equation}\label{defM}
M:=W_1U^{\otimes (k-l)}\otimes (\Uma)^{\otimes l-1}W_2
\end{equation}
and define $a,b\in \C^{\nn^{k-1}}$ by
\begin{equation}\label{eq:defa}
a_i=\sum_{t=1}^n [Uv]_t [W_1 x]_{(i_1,\ldots,t,\ldots,i_{k-1})} \;\;\;\;\textnormal{and}\;\;\;\;b_i=\sum_{t=1}^n v_t \overline{[W_2y]}_{(i_1,\ldots,t,\ldots,i_{k-1})},
\end{equation}
where in both terms $t$ is the $l^{th}$ index, so that 
\begin{equation*}
|\langle x,T_l y\rangle| = |\gamma| |\langle a,Mb\rangle  |\leq |\gamma|  \|M\|\|a\|_2\|b\|_2. \label{equation:Mw}
\end{equation*}
For $M$ we have the bound $\|M\|\leq (1+|\gamma|)^{(l-1)}$.
Since $v$ is uniformly distributed on $S^{n-1}$, for each $i$, and all $r>0$,
\begin{equation*}
\P\left(|a_{i}|>\frac{\log^r n}{\sqrt{n}}\left(\sum_{t=1}^n  |x_{(i_1,\ldots,t,\ldots,i_{k-1})}|^2\right)^\oh\right) \leq e^{-\oh \log^{2r}n},
\end{equation*}
and the analogous bound holds for each $b_i$; see, for example, Lemma B.1 in~\cite{SST06}. 
So with probability at least $1-2n^{k-1}e^{-\oh \log^{2r}n}$,
\begin{eqnarray*}
\|a\|_2\|b\|_2 &\leq& \frac{\log^{2r} n}{n} \left( \sum_{i\in\nn^{k-1}} \sum_{t=1}^n  |x_{(i_1,\ldots,t,\ldots,i_{k-1})}|^2   \right)^{1/2} \left( \sum_{i\in\nn^{k-1}} \sum_{t=1}^n  |y_{(i_1,\ldots,t,\ldots,i_{k-1})}|^2\right)^{1/2}\\
&=& \frac{\log^{2r} n}{n} \|x\|_2\|y\|_2.
\end{eqnarray*}
We now have a probabilistic bound on $|\langle x ,T_l y\rangle|$ for arbitrary $x$ and $y$
and $l=1,\ldots,k$.
Thus, with probability at
least $1-2kn^{k-1}e^{-\oh \log^{2r}n}$,
\begin{equation*}
|\langle x,T y\rangle|\leq k(1+|\gamma|)^{(k-1)}\frac{\log^{2r} n}{n}\|x\|_2\|y\|_2
\end{equation*}
for any fixed $x$ and $y$.
Since $\|\uj\|_2=1$ and $\|R(z)\|,\|R'(z)\|<\eta^{-1}$, we have that the first term in~\eqref{equation:twotermsb}
is bounded by
\begin{equation*}
2k|\gamma|(1+|\gamma|)^{k}\eta^{-2}\frac{\log^{2r} n}{n}
\end{equation*}
with probability at least $1-2kn^{k-1}e^{-\oh \log^{2r}n}$.
The Cauchy-Schwartz inequality gives that the second term in~\eqref{equation:twotermsb}
satisfies the same bound with the same probability.
We use a similar calculation for the third term in~\eqref{equation:twotermsb}
to obtain
\begin{equation}\label{eq:singleR}
|(\wj-\wj')^*R(z)\Wos\wj|\leq k \frac{\log^{2r}n}{n}\|M\| \|R(z)\|\leq  k\eta^{-1} (1+|\gamma|)^{(k-1)}\frac{\log^{2r}n}{n}
\end{equation}
with probability at least $1-kn^{k-1}e^{-\oh \log^{2r}n}$.
Thus, using the worst-case bound of $2\eta^{-1}$ on the event with small probability,
the bound $|\gamma|<2$ and the assumption $r = 2$, for large $n$ we obtain
\begin{eqnarray}
\E |\eqref{equation:termtobound}|^2&
\leq& \left(8|\gamma|k(1+|\gamma|)^{k}\eta^{-2}\frac{\log^{2r} n}{n}+2k\eta^{-1}(1+|\gamma|)^{k}\frac{\log^{2r} n}{n}\right)^2+10(2\eta^{-1})^2kn^{k}e^{-\oh \log^{2r}n}\nonumber\\
&\leq &\left( 4^{k} \frac{\log^{2r+1} n}{\eta^2 n}\right)^2\nonumber\\
&\leq &\left( \frac{ \log^{2r+1}n}{\eta^2 n^{1-2\cz}}\right)^2\nonumber
\end{eqnarray}
or
\begin{equation}\label{equation:B}
(\E(|f(U)-f(U')|^2)^{1/2}\leq \frac{ \log^{2r+1}n}{\eta^2 n^{1-2\cz}}.
\end{equation}
We also have the uniform bound
\begin{equation}
\|f\|_{\infty}\leq \frac{2}{\eta} \label{equation:A}
\end{equation}
for all $z\in\C^+$.

We now use Proposition 2.5 of~\cite{Cha07a}.
We use $K$ to denote the constant necessary to apply Chatterjee's result.
Using the bounds~\eqref{equation:B} and~\eqref{equation:A}, for large $n$ we have
\begin{equation*}
K\leq \frac{\rho\log^{4r+2}n}{\eta^4n^{1-2\cz}}
\end{equation*}
for an absolute constant $\rho>0$.
Now, for all $t>0$
\begin{equation*}
\P(|f(U)|>t)\leq 2\exp\left\{ \frac{-t^2 \eta^{-2}n^{1-2\cz}}{\rho \log^{2(r+2)}n} \right\},
\end{equation*}
so that, setting $r=2$,  and recalling the definition of $\eta$,
\begin{equation*}
\P\left(|f(U)|>\frac{1}{n^{\alpha}}\right)\leq 2 e^{-\log^{s} n}.
\end{equation*}
We obtain~\eqref{equation:maxdelta} by taking the union bound.
\end{proof}

\section{The Starting Point}

Recall that the first statement of Theorem~\ref{theorem:main} is that $|m_M(z)-m(z)|$ is small with high probability for $z$ having small imaginary part. 
Yet, when $z$ has small imaginary part, we do not have a good bound on $\|R(z)\|$. 
We, therefore, begin our argument with $\Im z=\O(1)$, and show that $|m_M(z)-m(z)|$ is small in this region. 
Then, in Section~\ref{section:cont} we use a continuity argument to incrementally decrease $\Im z$ to $\eta$ and conclude the proofs of the main theorems.

\begin{lemma}\label{lemma:constants}
There exist constants $0<c_M\leq C_M<\infty$ such that
\begin{equation}
|m_M(z)|\leq C_M
\end{equation}
and
\begin{equation}
\Im m_M(z)\geq c_M \sqrt{\kappa}\label{equation:lowerboundImm}
\end{equation}
for all $z\in \Omega$.
\end{lemma}

\begin{proof}
The first inequality holds because $m_M$ is analytic and $\Omega$ is a bounded region.
The second inequality holds again because $m_M$ is analytic and, by the Stieltjes inversion formula,
\begin{equation*}
f_M(x)=\frac{1}{\pi}\lim_{\omega\rightarrow 0^+}\Im m_M(x+i\omega),
\end{equation*}
and $f_M$ has square root singularities only at $\pm\lambda$.
\end{proof}

\begin{lemma}\label{lemma:startingpoint}
Assume the hypotheses of Theorem~\ref{theorem:main}.
For fixed $E\in\interval$, with probability at least
\begin{equation}
1-4n^{k}e^{-\log^s n}
\end{equation}
we have
\begin{equation*}
|m_M(E+i\cb)-m(E+i\cb)|=\O\left(\frac{1}{n^\alpha\kappa^{2}}\right).
\end{equation*}
\end{lemma}

\begin{proof}
Recall the indexing for $\U$ given in~\eqref{eq:index} and 
let $\nnku$ denote the subset of $\nn^k$ consisting of $k$-tuples of $k$ unique integers. 
Note that $\uj$ is identically distributed for all $j\in\nnku$, but $u_{(1,1,3,\ldots,k)}$, for example, has a slightly different distribution and will 
have to be treated separately. 
Since $\uj$ is identically distributed for $j\in\nnku$, it follows that 
$\wj$, $\sum_{k\neq j}Q_{k,k} P w_k w_k^* P$, $\Rj(z)$, and hence
\begin{equation*}
\wjs P\Rj(z)P \wj,
\end{equation*}
are also identically distributed for all $j\in\nnku$. 

We now define several more quantities:
\begin{equation*}
\delta_j(z):=\wjs P\Rj(z)P \wj-\E\wjs  P\Rj(z)P\wj
\end{equation*}
and
\begin{equation*}
D(z):=\E w_{(1,\ldots,k)}^* PR_{(1,\ldots,k)}(z)P w_{(1,\ldots,k)}
\;\;\;\textnormal{and}\;\;\;\delta(z):=\max|\delta_j(z)|,
\end{equation*}
where in the definition of $D(z)$ we have chosen ${(1,\ldots,k)}$ as an arbitrary element in $\nnku$. 
Note that a  bound on $|\delta(z)|$ was obtained in Lemma~\ref{lemma:concentration}; 
the proof  will conclude by applying that bound. 

For the proof of Lemma~\ref{lemma:startingpoint}, 
we set $z:=E+i\cb$ so that we have the simple bound $\|R(z)\|\leq 1/\cb$. 
In the following we  use that if $A$ is an $n\times n$ matrix, $q\in \C^n$ and both 
$A$ and $A+qq^*$ are invertible, then 
\begin{equation}
q^*(A+qq^*)^{-1}=\frac{1}{1+q^* A^{-1}q}q^*A^{-1}, 
\end{equation}
which one may verify directly. 
By writing
\begin{equation*}
P  \W  Q \Wstar P =\sumnn \Qjj P  \wjs\wj P,
\end{equation*}
the same calculations as equations $(21)$ through $(23)$ of~\cite{EF13} yield
\begin{equation}\label{eq:fordenom}
\wjs P R(z)P\wj= \frac{1}{1+\Qjj \wjs  PR_j(z)P\wj}\wjs PR_j(z)P\wj
\end{equation}
and
\begin{equation}\label{eq:implicitfinite}
 m(z)=\frac{-1}{z}\frac{1}{n^k}\sum_{j=1}^{n^k} \frac{1}{1+\Qjj \wjs PR_j(z)P\wj}.
\end{equation}
Using~\eqref{eq:fordenom}, we have that  $|1+\Qjj \wjs PR_j(z)P\wj|=\O(1)$ for all $j$.
Since $\Im m(z)=\O(1)$, we also have $|m(z)|=\O(1)$.
Thus, there exists a constant $B(z)$ satisfying $1\geq B(z)=\O(1)$ such that
\begin{equation}
|1+\wjs  P R(z)P \wj|,|m(z)|\geq B(z). \label{equation:lowerB}
\end{equation}
Since 
\begin{equation*}
 \frac{|\nn^k\backslash \nnku|}{n^k}\leq n^{-1/2}
\end{equation*}
is much smaller than $\delta(z)$ and $|D(z),|\wjs PR_j(z)P\wj|=\O(1)$ for all $j$, in the following we  
absorb these terms into error terms in $\delta(z)$. 
We then obtain 
\begin{eqnarray*}
D(z)m(z)
&=& \frac{-1}{z}\onk \sumnn \frac{D(z)}{1+\Qjj \wjs PR_j(z)P\wj}\\
&=&\frac{-1}{z}\onk\sumu \frac{\wjs PR_j(z)P\wj}{1+\Qjj \wjs PR_j(z)P\wj}
+\frac{1}{z}\onk \sumu \frac{\delta_j(z)}{1+\Qjj \wjs PR_j(z)P\wj}\\
&&-\frac{1}{z}\onk \sumc \frac{\wjs PR_j(z)P\wj}{1+\Qjj \wjs PR_j(z)P\wj}-\frac{1}{z}\onk \sumc \frac{D(z)}{1+\Qjj \wjs PR_j(z)P\wj}\\
&&+\frac{1}{z}\onk \sumc \frac{\wjs PR_j(z)P\wj}{1+\Qjj \wjs PR_j(z)P\wj}\\
&=&\frac{-1}{z}\onk \sumnn \frac{\wjs PR_j(z)P\wj}{1+\Qjj \wjs  PR_j(z)P\wj}+\oz\frac{\O (\delta(z))}{B(z)}  \\
&=&\frac{-1}{z}\onk \sumnn \wjs  P R(z)P\wj+\oz\frac{\O (\delta(z))}{B(z)}  \\
&=&\frac{-1}{z} \tr  \Wstar  P R(z)P \W+\oz\frac{\O (\delta(z))}{B(z)}.
\end{eqnarray*}
Now, since $[R(z)]_{j,j}=-1/z$ when $\Pjj=0$,
\begin{eqnarray*}
\frac{1}{n^k}\tr  \Wstar  P R(z)P \W 
&=& \frac{1}{n^k}\tr  P R(z)P= \onk\sumnn 1(\{\Pjj=1\})[R(z)]_{j,j}\\
&=& \onk \sum_{j=1}^{n^k} [R(z)]_{j,j}-\onk\sumnn 1(\{\Pjj=0  \})[R(z)]_{j,j}=\frac{1}{n^k}\tr R(z)+\frac{(1-p)}{z}.
\end{eqnarray*}
Thus,
\begin{equation*}
D(z)m(z)= \frac{-1}{z}\left(m(z)+\frac{1-p}{z}\right)+\oz \frac{\O (\delta(z))}{B(z)},  
\end{equation*}
and, by~\eqref{equation:lowerB},
\begin{equation*}
D(z)=\moz \left(1+\frac{1-p}{zm(z)}\right)+\oz\frac{\O (\delta(z))}{B^2(z)} .
\end{equation*}
Next,
\begin{eqnarray}
m(z)&=&\frac{-1}{z}\frac{1}{N}\sumnn  \frac{1}{1+\Qjj \wjs PR_j(z)P\wj}\nonumber\\
&=&\onk\sumnn\frac{-1}{z+z\Qjj(D+\delta_j(z))}\nonumber\\
&=&\onk\sumnn\frac{-1}{z+z\Qjj\left(\moz\left(1+\frac{1-p}{zm(z)}\right)+\frac{\O (\delta(z))}{B^2(z)}+\O (\delta_j(z))\right)}\nonumber\\
&=&\onk\sumnn\frac{-1}{z+z\Qjj\left(\moz\left(1+\frac{1-p}{zm(z)}\right)+\frac{\O (\delta(z))}{B^2(z)}\right)}.\label{equation:withO}
\end{eqnarray}
We have
\begin{equation*}
0<\Im D(z)=\O(1),
\end{equation*}
so that if $\delta(z)=\o(1)$, then
\begin{equation}
\Im \moz \left(1+\frac{1-p}{zm(z)}\right)=\O(1),
\end{equation}
and hence we may choose $B(z)$ to also satisfy $|z||1-\oz(1+\frac{1-p}{zm(z)})|\geq B(z)$.
Then
\begin{eqnarray*}
\eqref{equation:withO}&=&\onk\sumnn\frac{-1}{z-\Qjj\left(1+\frac{1-p}{zm(z)}\right)}+\frac{\O (\delta(z))}{B^4(z)}\\
&=& \frac{-(1-q)}{z}-\frac{q}{z-\left(1+\frac{1-p}{zm(z)}\right)}+\frac{\O (\delta(z))}{B^4(z)}.
\end{eqnarray*}
The solutions to the equation
\begin{equation}\label{eq:implicit}
m(z)= \frac{-(1-q)}{z}-\frac{q}{z-\left(1+\frac{1-p}{zm(z)}\right)}+\Lambda
\end{equation}
are
\begin{small}\begin{eqnarray*}
m_\Lambda(z)=\frac{2-p-q-z+z(z-1)\Lambda}{2(z^2-z)} \hspace{4in}\\
\pm\frac{\sqrt{ (z-2+p+q)^2-4(z^2-z)+z(z-1)\Lambda^2-2((z-2+p+q)(z^2-z)-4(z^2-z)(1-p))\Lambda}}{2z(z-1)}
\end{eqnarray*}
\end{small}
and only the solution with addition is the Stieltjes transform of a measure, which can be seen by considering large values of $z$.
The inequality
\begin{equation*}
|\sqrt{a}-\sqrt{a+b}|\leq C\frac{|b|}{\sqrt{|a|+|b|}},
\end{equation*}
which holds for an absolute constant $C$ for all $a,b\in \C$,  gives
\begin{equation*}
|m_M(z)-m(z)|=\O\left(\frac{\delta(z)}{\kappa^{2}}\right).
\end{equation*}
The lemma now follows by applying Lemma~\ref{lemma:concentration}.
\end{proof}

\section{Continuity Argument and Proofs of the Theorems}\label{section:cont}

\begin{lemma}\label{lemma:continuity}
There exist constants $C,\ca,\cb>0$ such that if
\begin{equation}
|m(E+i\ez)-m_M(E+i\ez)|\leq \ca \sqrt{\kappa}\label{equation:condition}
\end{equation}
with probability $1-P(n)$, then with probability at least $1-P(n)-2n^ke^{-\log ^2 n}$,
\begin{equation*}
|m(E+i(\ez-n^{-2}))-m_M(E+i(\ez-n^{-2}))|\leq \frac{C}{n^{\alpha}\kappa^{2}},
\end{equation*}
provided $\cb\geq \ez,\ez-n^{-2}\geq \eta$.
\end{lemma}

We do not give all the details for the proof of Lemma~\ref{lemma:continuity};
the argument follows the general idea of the proof of Theorem 1.1 of \cite{ESY09a}, and more specifically
Lemma 3.16 of~\cite{EF13}.

\begin{proof}
The proof requires lower bounds on $|m(\step)|$ and $|z-(1+\frac{1-p}{zm(\step)})|$.
For the first term, we use that $|\frac{d}{d\eta_0}m(E+i\eta_0)|\leq n$ for all $E+i\eta_0\in \Omega$.
Therefore, if $|m(E+i\ez)-m_M(E+i\ez)|<\ca \sqrt{\kappa}$ for small enough $\ca$,  by~\eqref{equation:lowerboundImm} we also have
$|m(\step)|>\oh c_M\sqrt{\kappa}$.
By inequality~\eqref{equation:lowerboundImm}, there exists $\cc>0$ such that for all sufficiently small $\cb$,
\begin{equation*}
\frac{(1-p)\Re z\Im m_M(z)}{|zm_M(z)|^2}>\cc\sqrt{\kappa}
\end{equation*}
for all $z\in\Omega$.
We assume that $\cb$ is small enough so that
$(1-p)\frac{\Re m_M(z)\Im z}{|zm_M(z)|^2}<\oh\cc\sqrt{\kappa}$ holds
for all $z\in\Omega$.
We then have
\begin{equation*}
\left|z-\left(1+\overmM\right)\right|\geq \left|\Im z-\Im\overmM\right|\geq \frac{(1-p)\Re z\Im m_M(z)}{|zm_M(z)|^2}-\oh\cc\sqrt{\kappa}\geq\oh \cc\sqrt{\kappa}.
\end{equation*}
Therefore, if $|m(\step)-m_M(\step)|$ is sufficiently small, $\left|z-\left(1+\overm\right)\right| \geq \frac{1}{4}\cc\sqrt{\kappa}$.
We then set $B(z)=\frac{1}{4}\cc\sqrt{\kappa}$ and follow the proof of Lemma~\ref{lemma:startingpoint}.
\end{proof}

\begin{proof}[Proof of Theorem~\ref{theorem:main}]
By Lemma~\ref{lemma:startingpoint}, condition~\eqref{equation:condition} is satisfied for $z=E+i\cb$
for a fixed $E$.
We apply Lemma~\ref{lemma:continuity} iteratively at most $n^2$ times to obtain the desired bound for the point $E+i\eta$.
The derivative of $m(z)$ with respect to $E$ is bounded uniformly on $\Omega$ by $n$.
We discretize $\{z\in\C:\;z=E+i\eta_0,\;E\in\interval\}$ to a grid of at most $n$ equally spaced points.
If $|m(z)-m_M(z)|<n^{-\alpha}\kappa^{-2}$ for all the points in the grid, then~\eqref{equation:maineqprop} follows by taking a union bound.
Inequality~\eqref{equation:maineqthm} now follows from inequality~\eqref{equation:maineqprop} by the argument given to prove Corollary 4.2 in~\cite{ESY09}.
\end{proof}

\begin{proof}[Proof of Theorem~\ref{theorem:second}]
The proof of Theorem~\ref{theorem:second} requires only a simple adjustment to the proof of Theorem~\ref{theorem:main}. 
In particular, Lemma~\ref{lemma:concentration} simplifies in that the the exponent $k$ is now one and operator $M$ is now unitary.
The other necessary lemmas and the final proof are then unchanged.
\end{proof}

\vspace{.5cm}

\begin{center}
ACKNOWLEDGEMENT
\end{center}

\vspace{.1cm}

B. Farrell  was partially supported by Joel A. Tropp under ONR awards N00014-08-1-0883
and N00014-11-1002 and a Sloan Research Fellowship.
R.R. Nadakuditi was partially supported by an ONR Young Investigator Award  N000141110660, 
an AFOSR Young Investigator Award FA9550-12-1-0266, a ARO MURI grant W911NF-11-1-0391 and NSF CCF–1116115

\def\cprime{$'$} \def\cprime{$'$} \def\cprime{$'$}

\end{document}